\documentclass[12pt]{article} 
\newcommand\useplain[1]{#1}

\usepackage{amsthm}
\usepackage{fullpage}
\newtheorem{theorem}{Theorem}
\newtheorem{lemma}[theorem]{Lemma}

\usepackage{amssymb, amsmath}
\usepackage{hyperref}
\usepackage{graphicx}

\newcommand\ex{\ensuremath{\mathrm{ex}}}
\newtheorem{cl}{Claim}
\newtheorem{claim}[cl]{Claim}

\useplain{\title{On the Tur\'an number of forests}}

\begin{document}

\author{
Bernard Lidick\'{y}
\thanks{Department of Mathematical Sciences,
University of Illinois at Urbana-Champaign, Urbana, Illinois 61801, USA {\tt lidicky@illinois.edu}.}
\and
Hong Liu
\thanks{Department of Mathematical Sciences,
University of Illinois at Urbana-Champaign, Urbana, Illinois 61801, USA {\tt hliu36@illinois.edu}.}
\and
Cory Palmer
\thanks{Department of Mathematical Sciences,
University of Illinois at Urbana-Champaign, Urbana, Illinois 61801, USA {\tt ctpalmer@illinois.edu} Research partially supported by OTKA NK 78439.}
}

\maketitle

\begin{abstract}
The Tur\'an number of a graph $H$, $\ex(n,H)$, is the maximum number of edges in a graph on $n$ vertices which does not have $H$ as a subgraph.
We determine the Tur\'an number and find the unique extremal graph for forests consisting of paths when $n$
is sufficiently large.
This generalizes a result of Bushaw and Kettle [%
Combinatorics, Probability and Computing 20:837--853, 2011].
We also determine the Tur\'an number and extremal graphs for forests consisting of stars of arbitrary order.
\end{abstract}

\section{Introduction}

Notation in this paper is standard. For a graph $G$ let $E(G)$ be the set of edges and $V(G)$ be the set of vertices.
The \emph{order} of a graph is the number of vertices. The number of edges of $G$ is denoted $e(G) = |E(G)|$.
For a graph $G$ with subgraph $H$, the graph $G-H$ is the induced subgraph on vertex set $V(G)\setminus V(H)$ i.e. $G[V(G)\setminus V(H)]$.
For $U \subset V(G)$ we define $N(U)$ to be the set of vertices in $V(G)\setminus U$ that have a neighbor in $U$. While the \emph{common neighborhood} of $U \subset V(G)$ is
the set of vertices in $V(G) \setminus U$ that are adjacent to every vertex in $U$. We denote the degree of a vertex $v$ by $d(v)$ and the minimum degree in a graph by $\delta(G)$.
A \emph{universal vertex} in $G$ is a vertex that is adjacent to all other vertices in $G$.
A \emph{star forest} is a forest whose connected components are stars and a \emph{linear forest} is a forest whose connected components are paths.
A path on $k$ vertices is denoted $P_k$ and a star with $k+1$ vertices is denoted $S_k$. Superscript is used to denote the index of a particular graph in a set of graphs.
Let $k \cdot H$ denote the graph of the disjoint union of $k$ copies of the graph $H$.

The \emph{Tur\'an number}, $\ex(n,H)$, of a graph $H$ is the maximum number of edges in a graph on $n$ vertices which does not contain $H$ as a subgraph. The problem of determining Tur\'an numbers is one of the cornerstones of graph theory. The traditional starting point of extremal graph theory is a theorem of Mantel (see e.g. \cite{bib-bollobas-MGT}) that the maximum number of edges in a triangle-free graph is $\lfloor \frac{n^2}{4} \rfloor$. Tur\'an \cite{Turan41,Turan54} generalized this result to find the extremal graph of any complete graph. In particular he showed that $\ex(n,K_{r}) = \left(\frac{r-2}{r-1}\right)\cdot\frac{n^2}{2}$.  The Erd\H{o}s-Stone-Simonovits Theorem \cite{ErSt46,ErSim66} states that asymptotically Tur\'an's construction is best-possible for any $r$-chromatic graph $H$ (as long as $r>2$). More precisely $\ex(n,H) = \left(\frac{r-2}{r-1}\right)\cdot\frac{n^2}{2} + o(n^2)$.
Thus when $H$ is bipartite the Erd\H{o}s-Stone-Simonovits only states that $\ex(n,H) = o(n^2)$. The Tur\'an number of bipartite graphs includes many classical theorems. For example
for complete bipartite graphs K\H{o}v\'ari-S\'os-Tur\'an Theorem \cite{K-S-T1954} gives $\ex(n,K_{s,t}) = O(n^{2-1/s})$ (also see e.g. \cite{ErReSos66,Brown66,Fu96}) and for even cycles Bondy and Simonovits \cite{Bondy1974} have $\ex(n,C_{2k}) \leq n^{1+1/k}$.

In 1959, Erd\H{o}s and Gallai~\cite{E-G} determined the Tur\'an number for paths. We state the theorem here as it is an important tool in our proofs.

\begin{theorem}[\cite{E-G}]\label{E-G}
For any $k, n>1$, $\ex(n,P_{k})\le \frac{k-2}{2}n$, where equality holds for the graph of disjoint copies of $K_{k-1}$.
\end{theorem}

A well-known conjecture of Erd\H{o}s and S\'os \cite{Er46} states that the Tur\'an number for paths is enough for any tree i.e. a graph $G$ on $n$ vertices and more than $\frac{k-2}{2}n$ edges contains any tree on $k$ vertices. A proof of the Erd\H{o}s-S\'os Conjecture for large trees was announced by Ajtai, Koml\'os, Simonovits and Szemer\'edi.

A natural extension of the problem is the determination of the Tur\'an number of forests. Erd\H{o}s and Gallai \cite{E-G} considered the graph $H$ consisting of $k$ independent edges (note that $H$ is a linear forest) and found $\ex(n,H) = \max\{{k-1 \choose 2}+(k-1)(n-k+1),{2k-1 \choose k}\}$. When $n$ is large enough compared to $k$, the extremal graph attaining this bound is obtained by adding $k-1$ universal vertices to an independent set of $n-k+1$ vertices.
This construction clearly does not contain $k$ independent edges as every such edge must include at least one of the universal vertices.
This construction forms a model for the constructions presented throughout the paper. Brandt \cite{Brandt94} generalized the above result by proving that any graph $G$ with  $e(G) > \max\{{k-1 \choose 2}+(k-1)(n-k+1),{2k-1 \choose k}\}$ contains every forest on $k$ edges without isolated vertices.

Recently, Bushaw and Kettle~\cite{B-K} found the Tur\'an number and extremal graph for the linear forest with
components of the same order $l>2$. When $l=3$ this proves a conjecture of Gorgol \cite{Gorgol2011}. We generalize this theorem by finding the Tur\'an number and extremal graph for arbitrary linear forests.

\begin{theorem}\label{thm:paths}
Let $F$ be a linear forest with components of order $v_1,v_2,\dots,v_k$.
If at least one $v_i$ is not $3$, then for $n$ sufficiently large,
$$\ex(n,F)=\left(\sum_{i=1}^k \left\lfloor \frac{v_i}{2} \right\rfloor -1\right)\left(n-\sum_{i=1}^k\left\lfloor \frac{v_i}{2} \right\rfloor+1\right)+{\sum_{i=1}^k\left\lfloor \frac{v_i}{2} \right\rfloor-1\choose 2}+c,$$
where $c=1$ if all $v_i$ are odd and $c=0$ otherwise. Moreover, the extremal graph is unique.
\end{theorem}

Notice that the theorem avoids the case of linear forest with every component of order three.
This case was solved by Bushaw and Kettle~\cite{B-K}. We will also solve it as a special case
of a star forest handled by Theorem~\ref{thm:stars} as $P_3$ is the star with two leaves $S_2$.
We prove Theorem \ref{thm:paths} in Section \ref{sec:linear} and describe the unique $F$-free graph on $n$ vertices with $\ex(n,F)$ edges.

Another motivation is the following conjecture of Goldberg and Magdon-Ismail~\cite{GoldbergM11}:
\begin{center}
\begin{minipage}{35em}
\emph{
Let $F$ be a forest with $k$ components, then every graph $G$ with at least $e(F) + k$ vertices and average
degree $> e(F) - 1$ contains $F$ as a subgraph.}
\end{minipage}
\end{center}
This is a natural generalization of
the Erd\H{o}s-S\'{o}s Conjecture, however it is false. 
From Theorem~\ref{thm:paths}, some simple calculation shows that the conjecture fails for linear
forests with at least two components of even order.

We also investigate the other extreme case when each component is a star and we determine the Tur\'an number find all the extremal graphs.

\begin{theorem}\label{thm:stars}
Let $F=\bigcup_{i=1}^k S^i$ be a star forest where $d_i$ is the maximum degree of $S^i$ and
 $d_1\ge d_2\ge \cdots\ge d_k$.
For $n$ sufficiently large, the Tur\'an number for $F$ is
$$\ex(n,F)=\max_{1\le i\le k}\left\{(i-1)(n-i+1)+{i-1\choose 2}+\left\lfloor\frac{d_i-1}{2}(n-i+1)\right\rfloor\right\}.$$
\end{theorem}

Note that for a single component, Theorem \ref{thm:stars} describes the Tur\'an number of a star i.e. the maximum number of edges in a graph of fixed maximum degree.
We prove Theorem~\ref{thm:stars} and characterize all extremal graphs in Section~\ref{sec:stars}. The last section contains a result about forests with small components.
In the proofs of Theorem \ref{thm:paths} and Theorem \ref{thm:stars} we make no attempt to minimize the bound on $n$.

\section{Linear forests}\label{sec:linear}

We first consider the Tur\'an problem for a linear forest. Throughout this section, unless otherwise specified, $F$ is a linear forest i.e. $F=\bigcup_{i=1}^k P^i$, such that $P^i$ is a path on $v_i$ vertices and $v_1\ge v_2\ge \dots \ge v_k\ge 2$.

Bushaw and Kettle \cite{B-K} determined the Tur\'an number and extremal graph for forests of paths of the same order.

\begin{theorem}[\cite{B-K}]\label{disjoint-union-path}
Let $F$ be a linear forest such that each component has order $\ell$.
For $n$ large enough if $\ell=3$, then
$$\ex(n,k\cdot P_3)={k-1\choose 2}+(n-k+1)(k-1)+\left\lfloor\frac{n-k+1}{2}\right\rfloor.$$
if $\ell\geq 4$, then
$$\ex(n, k\cdot P_{\ell})={k\left\lfloor\frac{\ell}{2}\right\rfloor-1 \choose 2}+\left(k\left\lfloor\frac{\ell}{2}\right\rfloor-1\right)\left(n-k\left\lfloor\frac{\ell}{2}\right\rfloor+1\right)+c,$$
where $c=1$ if $\ell$ is odd, and $c=0$ if $\ell$ is even.
\end{theorem}

\begin{figure}[ht]
\begin{center}
\useplain{\includegraphics{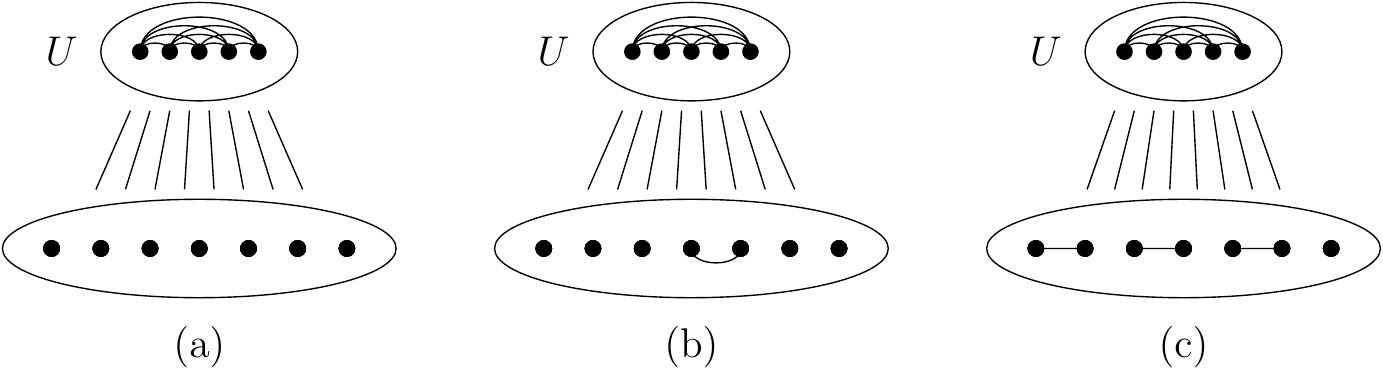}}
\end{center}
\caption{$G_F(n)$, (a) is the case where at least one path $F$ is of even order, (b) is the case where all paths in $F$ are of odd order. Figure (c) is the extremal graph for $k \cdot P_3$.}
\label{fig-gf}
\end{figure}

The extremal graph for $k$ copies of $P_3$ is a set of $k-1$ universal vertices and a maximal matching among the $n-k+1$ remaining vertices. See Figure~\ref{fig-gf}(c). We note that that this is the extremal graph for $k\cdot S_2$ in Section~\ref{sec:stars}.
The extremal graph for longer paths is in Figure~\ref{fig-gf}(a) and (b) with $k\left\lfloor\frac{l}{2}\right\rfloor-1$ universal vertices.

Let $F$ be a linear forest where at least one $v_i$ is not $3$. Define $G_F(n)$ to be the graph on $n$ vertices with a set
$U$ of $\left( \sum_{i=1}^k \lfloor v_i/2\rfloor \right) -1$ universal vertices together with a single edge in $G_F(n) - U$ if each $v_i$ is odd or $n-|U|$ independent vertices otherwise (see Figure~\ref{fig-gf}).
Observe that $G_F(n)$ is $F$-free. Indeed, any path $P^i$ in $G_F(n)$ uses at least $\lfloor v_i/2\rfloor$ vertices from $U$, and $|U|<\sum_i \lfloor v_i/2\rfloor$.
When every component is the same, then $G_F(n)$ is exactly the extremal graph given by Bushaw and Kettle \cite{B-K}.
We show that $G_F(n)$ is the unique extremal graph for a linear forest $F$.

First, let us consider the base case when $F$ consists of only two paths and by Theorem~\ref{disjoint-union-path} we may assume they are of different lengths.

\begin{theorem}\label{2-paths}
Suppose $F=P_a\cup P_b$, with $a>b\ge 2$. Let $G$ be any $F$-free $n$-vertex graph with $n\ge {a\choose a/2}^2a^4b^4$. Then $e(G)\le e(G_F(n))$, with equality only when $G\simeq G_F(n)$.
\end{theorem}

We use a standard trick to reduce the problem to graphs with large minimum degree.

\begin{lemma}\label{large-minimum-degree}
Suppose $F=P_a\cup P_b$, with $a>b\ge 2$. Let $G$ be any $F$-free $n$-vertex graph with $n\ge {a\choose a/2}a^2b^2$ and $\delta(G)\ge \lfloor a/2\rfloor+\lfloor b/2\rfloor-1$. Then $e(G)\le e(G_F(n))$, with equality only when $G\simeq G_F(n)$.
\end{lemma}

We first show how Lemma~\ref{large-minimum-degree} implies Theorem~\ref{2-paths} and we give the proof of the lemma afterwards.

\begin{proof}[Proof of Theorem~\ref{2-paths} using Lemma~\ref{large-minimum-degree}.]
Suppose $G$ is an extremal graph for $F$ on $n>{a\choose a/2}^2a^4b^4$ vertices and $e(G)\ge e(G_F(n))$.
We start by removing vertices of small degree from $G$.
Suppose that there exists a vertex $v$ in $G$ with  $d(v)<\delta(G_F(n))=\lfloor a/2\rfloor+\lfloor b/2\rfloor-1$.
Let $G^n = G$.  We then define $G^{n-1} = G^n - \{v\}$.  
We keep constructing $G^{i-1}$ from $G^i$ by removing
a vertex of degree less then $\delta(G_F(i))$. 
The process continues while $\delta(G^i)<\delta(G_F(i))$.
Notice that
\[e(G^{i-1}) - e(G_F(i-1)) \geq e(G^{i}) - e(G_F(i)) + 1.\]
Hence the process terminates after $n-\ell$ steps. We get $G^{\ell}$ with $\delta(G^{\ell})\ge \delta(G_F(\ell))$ and
\[{\ell\choose 2}\ge e(G^{\ell})\ge e(G_F(\ell))+n-\ell=(\lfloor a/2\rfloor+\lfloor b/2\rfloor-1)\ell+n-\ell+O(a^2b^2),\]
this implies $\ell>\sqrt{2n}\ge {a\choose a/2}a^2b^2$. Since $e(G^{\ell})>e(G_F(\ell))$, Lemma~\ref{large-minimum-degree} then implies $F\subseteq G^{\ell}\subseteq G$, a contradiction.
\end{proof}

\begin{proof}[Proof of Lemma~\ref{large-minimum-degree}.]
Let $G$ be an extremal graph for $F$ with $\delta(G)\ge \left\lfloor a/2\right\rfloor+\lfloor b/2\rfloor-1$. First we show that there is a set $U_a$ of $\left\lfloor a/2\right\rfloor$ vertices
in which all the vertices share a large common neighborhood, in particular more than $a+b$ common neighbors.
Such a $U_a$ can be easily extended to a copy of $P_a$, which implies $G-U_a$ must be $P_b$-free.
However, $G-U_a$ is not $P_{b-2}$-free and hence we can find
a  set $U_b$ of  $\lfloor b/2\rfloor-1$ vertices that share a large common neighborhood.
Then we show that $G-U_a-U_b$ has at most one edge, which then implies $G \subseteq G_F(n)$.

Because $e(G)\ge e(G_F(n))>\ex(n,P_a)$ we have $P_b \subseteq P_a\subseteq G$. The following claim is a special case of Lemma 2.3 in \cite{B-K}. We include a proof for the sake of completeness.

\begin{cl}\label{claim-large-neighbor}
Let $P_x$ be a path such that $x =a$ or $x=b$.
Every $P_x$ in $G$ contains at least $\left\lfloor\frac{x}{2}\right\rfloor$ vertices with common neighborhood larger than $\frac{1}{x{x\choose \lfloor x/2\rfloor}}n$.
\end{cl}

\begin{proof}
For simplicity we prove for $P_x = P_a$ and note that the proof for $P_x = P_b$ is the same (we do not use the fact that $a>b$).
The graph $G$ does not contain the forest $P_a \cup P_b$, so the graph $G - P_a$ is $P_b$-free. Thus $e(G - P_a) \leq \frac{b-2}{2}(n-a)$ by Erd\H os-Gallai (Theorem \ref{E-G}). Because $e(G) \geq e(G_F(n)$ the number of edges between $P_a$ and $G-P_a$ is at least 
$e(G_F(n)) - {a \choose 2} - \frac{b-2}{2}(n-a)$.

Let $n_0$ be the number of vertices in $G-P_a$ with at least $\lfloor \frac{a}{2} \rfloor$ neighbors in $P_a$. Therefore the number of edges between $P_a$
and $G-P_a$ is at most $n_0a + (n-a-n_0)(\lfloor \frac{a}{2} \rfloor -1)$. So
\[ e(G_F(n)) - {a \choose 2} - \frac{b-2}{2}(n-a) \leq n_0a + (n-a-n_0)\left(\left\lfloor \frac{a}{2} \right\rfloor -1\right) .\]

Substituting the value of $e(G_F(n))$ and solving for $n_0$ we get
\[n_0 \geq \frac{n/2+O(a^2b^2)}{ a-\lfloor\frac{a}{2}\rfloor+1 }.\]
There are ${a \choose \lfloor \frac{a}{2} \rfloor}$ sets of vertices of order $\lfloor \frac{a}{2} \rfloor$ in $P_a$. Thus there is some set with a common neighborhood of order at least 
\[\frac{n_0}{{a \choose \lfloor \frac{a}{2} \rfloor}} \geq \frac{n}{2a {a \choose \lfloor \frac{a}{2} \rfloor}}>a+b.\]
\end{proof}

As $G$ is not $P_a$-free, let $U_a$ be a set of vertices of large common neighborhood with $|U_a|=\lfloor \frac{a}{2}\rfloor$
whose existence is guaranteed by Claim~\ref{claim-large-neighbor}.

Observe that $G-U_a$ is $P_b$-free. Indeed, suppose there is a copy of $P_b$ in $G-U_a$ and $U_a$ can be extended to $P_a$ avoiding vertices of $P_b$, since $U_a$ has more than $a+b$ common neighbors.
Thus $F\subseteq G$ which is a contradiction.

We now distinguish three cases based on the value of $b$.
If $b=2$, then $G-U_a$ is $P_2$-free, namely $G-U_a$ is empty. Thus $G\subseteq G_F(n)$.

Suppose that $b=3$. Hence $G-U_a$ contains only isolated edges and vertices as it does not contain a copy of $P_3$. If $uv$ is an edge in $G-U_a$ then
both $u$ and $v$ have a neighbor in $U_a$ due by the minimum degree $\delta(G) \geq \lfloor a/2 \rfloor \geq 2$. Furthermore, at most one of $u$ and $v$ is not adjacent to all vertices of $U_a$, otherwise a graph
obtained from $G$ by removing the edge $uv$ and adding all edges between $\{u,v\}$ and $U_a$ has more edges
while it is still $F$-free, which contradicts the extremality of $G$. Let $z$ be a vertex in $U_a$. When $a$ is even, then $U_a\setminus z$ together with an edge in $G-U_a$ complete a copy of $P_a$, and $N(z)\cup \{z\}$ contains a $P_3=P_b$, see Figure~\ref{fig-b3}(a), thus $F\subseteq G$, contradiction. So $G-U_a$ is empty if $a$ is even. When $a$ is odd, then two edges in $G-U_a$ together with $U_a\setminus z$ complete a copy of $P_a$, and again $N(z)\cup \{z\}$ contains a $P_b$, see Figure~\ref{fig-b3}(b), thus $G-U_a$ has at most one edge. Therefore $G\subseteq G_F(n)$.

\useplain{\begin{figure}[ht]}
\begin{center}
\includegraphics{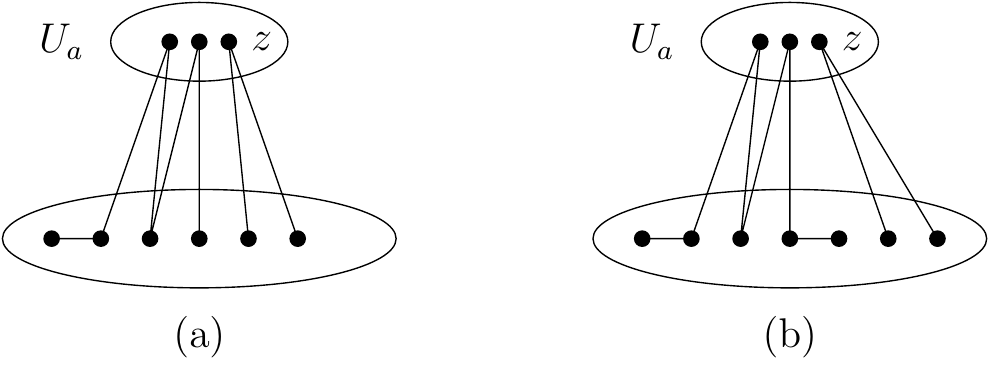}
\end{center}
\caption{Cases for $b=3$.}
\label{fig-b3}
\end{figure}

Now we may assume $b \geq 4$.
Let $|V(G-U_a)|=n'$. Denote by $G'$ the graph obtained by deleting $\lfloor a/2\rfloor$ universal vertices from $G_F(n)$, clearly $|V(G')|=n'$.

If $G-U_a$ is $P_{b-2}$-free, then $e(G-U_a)\le (\frac{b}{2}-2)n'<(\lfloor \frac{b}{2}\rfloor-1)n'-{\lfloor\frac{b}{2}\rfloor\choose 2}=e(G')$, which implies $e(G)<e(G_F(n))$ which is a contradiction.
Thus we may assume $G-U_a$ is not $P_{b-2}$-free. Take a maximum path, $P$, in $G-U_a$ and let $u$ be an end vertex. Since $G-U_a$ is $P_b$-free, $P$ has at most $b-1$ vertices. Then
$$d(u)-(|V(P)|-1)\ge \delta(G)-b+2\ge \lfloor a/2\rfloor+\lfloor b/2\rfloor-1-b+2=\lfloor a/2\rfloor-\lceil b/2\rceil+1\ge 1.$$
Thus $u$ has a neighbor, say $w$, not in $P$. Furthermore $w$ must be in $U_a$, since otherwise we could extend $P$ to a longer path in $G-U_a$. Then $P$ with $w$ and a neighbor of $w$ in $G-U_a-P$ form a $P_b$. By Claim~\ref{claim-large-neighbor} there is a set of $\lfloor b/2\rfloor$ vertices of $P_b$ with a common neighborhood of order at least $a+b$. Note that all vertices in this $P_b$ except $w$ are in $G-U_a$. Thus we find a set of vertices in $G-U_a$ of order $\lfloor b/2\rfloor-1$, call it $U_b$, with a large common neighborhood ($U_b$ is nonempty since $b\ge 4$). We are done if we show $G-U_a-U_b$ is empty if one of $a$ and $b$ is even or has at most one edge when they are both odd, as this implies $G\subseteq G_F(n)$.

\begin{claim} $G$ is connected. \end{claim}
Suppose for contradiction that $G$ is not connected. Then there exists
a connected component $C$ in $G$ not containing $U_a$. The minimum degree
$\left\lfloor a/2\right\rfloor+\lfloor b/2\rfloor-1 \geq b-1$ implies that it is possible
to find $P_b$ in $C$ (for example by a greedy algorithm).
It contradicts the claim that $G-U_a$ is $P_b$-free.
\qed

\newcommand\Hu{H - U}

Let $H$ be a graph and let $U \subset V(H)$ such that $U$ has at least $|U|+3$ common neighbors.
Furthermore, let $1 \leq c \leq |U|+1$.

\begin{claim}\label{cl:pt}
Let $v \in N(U)$ be an endpoint of $P_{t}$ in $\Hu$ where $1 \leq t \leq 3$.
Then $H$ contains  a path $P$ of order $2c-2+t$ such that $|P \cap U| = c-1$.
\end{claim}
Let $u \in U$ be a neighbor of $v$. Let $P'$ be a path in $H$ of order $2c-2$ starting at $u$ and avoiding vertices of $P_t$
and having at most $c-1$ vertices of $U$.
It can be obtained in a greedy way from $U$ and $N(U)\setminus V(P_t)$ by alternating between $U$ and $N(U)\setminus V(P_t)$.
The path $P_t$ can be used for extending $P'$ and finding a path of order $2c-2+t$.
See Figure~\ref{fig-claims}(a) for case $t=3$.
\qed

\useplain{\begin{figure}[ht]}
\begin{center}
\includegraphics{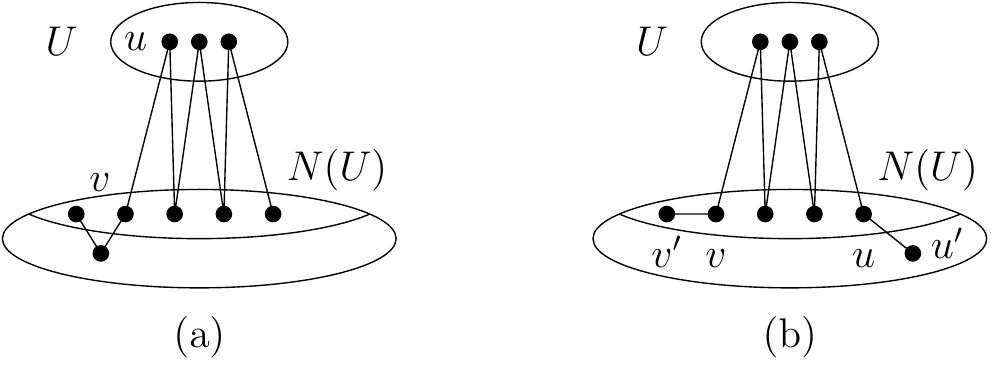}
\end{center}
\caption{Finding a long path.}
\label{fig-claims}
\end{figure}

\begin{claim}\label{cl:11}
If there exist nonadjacent $u,v \in N(U)$ both of degree at least one in  $\Hu$ then $H$ contains a path $P$ of order $2c+1$
such that $|P \cap U| = c-1$.
\end{claim}
If there is a common neighbor of $u$ and $v$ then the claim follows from Claim~\ref{cl:pt} because $u$ is an endpoint of a path of order three.
Let $u'$ and $v'$ be distinct neighbors of $u$ and $v$, respectively, in $\Hu$. A greedily obtainable path with end vertices
$u$ and $v$ avoiding $u'$ and $v'$ of order $2c-1$ can be extended by $u'$ and $v'$ to a path of order $2c+1$.
See Figure~\ref{fig-claims}(b).
\qed

\begin{claim}
There are at most two vertices in $N(U_b) \setminus U_a$ that are not adjacent to all vertices of $U_a$.
\end{claim}
Notice that if a vertex is not adjacent to all vertices of $U_a \cup U_b$ then it has degree at least one in $G - U_a - U_b$
because of the minimum degree condition. If there are at least three such vertices, Claim~\ref{cl:pt} or Claim~\ref{cl:11}
applies with $H = G - U_a$, $U_b$ in place of $U$ and $c = \lfloor\frac{b}{2}\rfloor$. It leads to existence
of a path of order $2\lfloor\frac{b}{2}\rfloor + 1$ in $ G - U_a$ which is a contradiction.
\qed

Let $U$ be $U_a \cup U_b$.
The previous claim together with the fact that vertices of $U_b$ have many common neighbors implies that there
is a common neighborhood of $U$ of order at least $\frac{n}{b\binom{b}{b/2}} - 2$.
Hence whenever we find in $G$ a path of order $a$ using at most $\lfloor\frac{a}{2}\rfloor - 1$ vertices of $U$
or a path of order $b$ using at most $\lfloor\frac{b}{2}\rfloor - 1$ vertices of $U$ then we can easily find the
other path of $F$.

\begin{claim}
Every vertex in $G-U$ is adjacent to a vertex of $U$.
\end{claim}
If a vertex $v$ is not adjacent to any vertex of $U$ then by connectivity and minimum degree condition of $G$ it is easy to find
 a vertex $u \in N(U)$ such that $u$ is an end vertex of a path of order 3.
Hence Claim~\ref{cl:pt} applies and gives a path of order $b$
using $\lfloor\frac{b}{2}\rfloor-1$ vertices of $U$ which is a contradiction.
\qed

Hence $N(U)$ are all the vertices of $G - U$.
Claim~\ref{cl:pt} implies that $G - U$ is $P_3$-free and together with Claim~\ref{cl:11} this implies that
$G - U$ contains at most one edge. Moreover, if $a$ or $b$ is even, Claim~\ref{cl:pt} implies that
$G - U$ is $P_2$-free. Hence $G - U$ contains no edges if $a$ or $b$ is even and contains at most one
edge if both $a$ and $b$ are odd.
Therefore, $G$ is a subgraph of $G_F(n)$.
\end{proof}

Now we are ready to prove Theorem~\ref{thm:paths}.

\begin{proof}[Proof of Theorem~\ref{thm:paths}.]
We proceed the proof by induction on $k$, the number of paths in $F$. The base case, $k=2$, is given by Theorem~\ref{2-paths}. Assume it is true for $1\le \ell\le k-1$. Pick $j$ such that $1\leq j\leq k$ and $F'=F-P^j$
is not $(k-1)\cdot P_3$. Let $G$ be an extremal graph for $F$. Take a $P^j$ in $G$, by Claim~\ref{claim-large-neighbor}, it contains a set $U_j$ of $\lfloor v_j/2\rfloor$ vertices, whose vertices have many common neighbors. Similarly, $G-U_j$ has to be $F'$-free, then by induction hypothesis $G-U_j\simeq G_{F'}$, thus $G\subseteq G_F(n)$. Hence $G_F(n)$ is the unique extremal graph.
\end{proof}


\section{Star forests}\label{sec:stars}

In this section we give a proof of Theorem~\ref{thm:stars}.

Let $F=\bigcup_{i=1}^k S^i$ be a star forest as in the statement of Theorem~\ref{thm:stars}.
Let $d_i$ be the maximum degree of $S^i$. Recall that $d_1\ge d_2\ge \cdots\ge d_k$.

We begin by describing the extremal graph for $F$.
Let $F(n,i)$ be a graph obtained by adding a set $U$ of $i-1$ universal vertices to an extremal graph, $H$, for $S^i$ on $n-i+1$ vertices.
See Figure~\ref{fig-sf}. Observe that $H$ is a $(d_i-1)$-regular graph if one of
$(d_i-1)$ and $n-i+1$ is even and $n$ large enough (see \cite{Simonovits1966}). If both are odd, then $H$ has exactly one vertex of degree $(d_i-2)$ and the remaining vertices have degree $(d_i-1)$. Therefore we have $e(H) = \left\lfloor\frac{d_i-1}{2}(n-i+1)\right\rfloor$.

\useplain{\begin{figure}[ht]}
\begin{center}
\useplain{\includegraphics{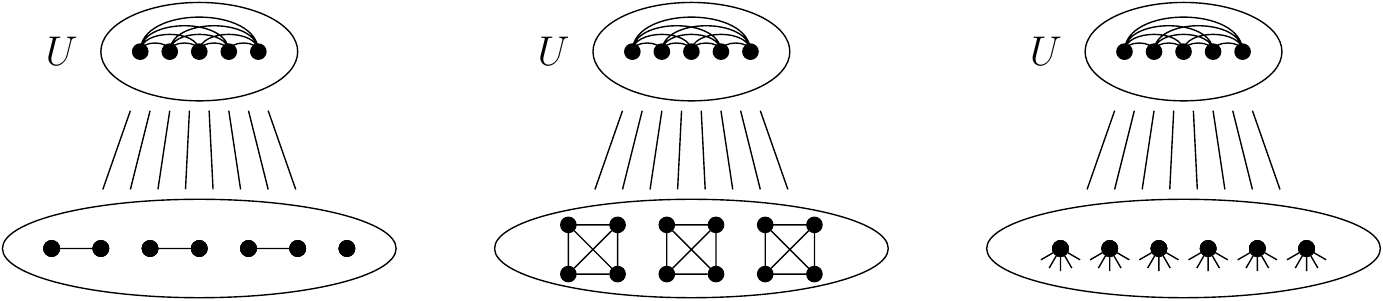}}
\end{center}
\caption{The extremal graph for  a star-forest.}
\label{fig-sf}
\end{figure}

Observe that for all $1\le i\le k$,  $F(n,i)$ is $F$-free. Indeed, each star $S^1,\ldots, S^{i-1}$ must have at least one vertex from $U$ and $S^i$ is not a subgraph of $F(n,i)-U$.

Throughout the proof, unless otherwise specified, $i$ is always the index maximizing the number of edges.
Notice that $e(F(n,i))$ is equal to the Tur\'an number claimed by the theorem, i.e.

$$e(F(n,i))=\left \lfloor \left(i-1+\frac{d_i-1}{2}\right)n-\frac{i-1}{2}(i+d_i-1) \right \rfloor.$$

Define $f_i=i-1+\frac{d_i-1}{2}$, namely $f_i$ is the coefficient of $n$, the leading term of $e(F(n,i))$.

\begin{cl}\label{claim1}
For any $j<i$, $f_j<f_i$.
\end{cl}

\begin{proof}
Suppose not. If $f_j>f_i$, then $j$ would be the index maximizing the number of edges, contradiction. Thus assume $f_j=f_i$, we are done if $e(F(n,j))-e(F(n,i))>0$, because this again contradicts the choice of $i$. Since $f_j=f_i$, the leading terms cancel. Furthermore $i-1+\frac{d_i-1}{2}=j-1+\frac{d_j-1}{2}$ implies
\begin{align}
d_j &=d_i+2(i-j). \label{eq:dj}
\end{align}
Thus we have:
\begin{align*}
e(F(n,j))-e(F(n,i)) &\geq \frac{i-1}{2}(i+d_i-1) - 1/2-\frac{j-1}{2}(j+d_j-1)\\
&= \frac{i-1}{2}(i+d_i-1) - 1/2-\frac{j-1}{2}(j+d_i+2(i-j)-1) && \text{  by \eqref{eq:dj}} \\
&= \frac{i-j}{2}(d_i+i-j)  - 1/2\\
&> 0\\
\end{align*}
\end{proof}

We use induction on $k$, the number of components of $F$. We prove the theorem in three cases distinguished by the index $i$ that maximizes the number of edges: (1) $i=k$, (2) $i\neq k$ and $i\neq 1$, and (3) $i=1$.

\noindent \textbf{Case (1)}: $i=k$

Let $F'=F-S^k$. Let $G$ be an extremal graph for $F$ with $n$ vertices. By the induction hypothesis, we have $\ex(n,F')=e(F(n,i'))$, where $i'$ is the index maximizing the number of edges. Since $i=k$, we have $i'<i$ and then by Claim~\ref{claim1}, thus
\begin{align}f_{i'}&<f_i=f_k. \label{eq:fi}\end{align}
Note that $F(n,i)$ is $F$-free, we have
$$\ex(n,F')=f_{i'}n+O(d_{i'})<f_kn+O(d_k)=e(F(n,k))\le e(G).$$
Thus $F'=S^1\cup S^2\cup \cdots\cup S^{k-1}\subseteq G$ by induction hypothesis.

It suffices to prove that there exists a vertex subset $U\subseteq V(G)$ of order $k-1$, such that every vertex in $U$ has linear degree, that is, $d(v)=\Omega(n)$. Indeed, if such a $U$ exists, then $G-U$ has to be $S^k$-free. Otherwise, say there is a $S^k$ in $G-U$, then we can get $F'$ using vertices in $U$ as centers and their neighbors in $G-U-S^k$ as leaves, which gives a copy of $F'\cup S^k=F$, which is a contradiction. 
Hence $G \subseteq F(n,k)$ as desired.

Now we prove such a $U$ exists. We know $F'\subseteq G$, namely there are $k-1$ disjoint stars in $G$. Take any one of them, say $S^j$, $1\le j\le k-1$. 
Notice that $G-S^j$ has to be $F'$-free, since otherwise a copy of $F'$ in $G-S^j$ together with $S^j$ yields a copy of $S^1\cup S^2\cup\cdots\cup S^{k-1}\cup S^j$.
Since $S^k\subseteq S^j$ we get that $F\subseteq G$, which is a contradiction. 
 Note that $e(G[S^j])\le {d_j+1\choose 2}$ and $e(G)\ge e(F(n,k))$. Let $e_{0}$ be the number of edges between $S^j$ and $G-S^j$. Then we have
\begin{align*}
e_{0}&= e(G)-e(G-S^j)-e(G[S^j])\\
&\ge e(F(n,k))-\ex(n,F')-{d_j+1\choose 2}\\
&\sim f_kn-f_{i'}n=\Omega(n) && \text{  by \eqref{eq:fi}.}\\
\end{align*}

Thus there is a vertex in $S^j$ with linear degree. Since this is true for every $j$ with $1\le j\le k-1$, take the one with linear degree from each star, these $k-1$ vertices form the desired set $U$.

This finishes the proof of Case (1).

\noindent \textbf{Case (2):} $i\neq k$ and $i\neq 1$

Let $F^*=S^i\cup S^{i+1}\cup\cdots\cup S^k$ and $F'=F-F^*$. Similarly if $i'$ is the index maximizing the number of edges for $F'$, then $f_{i'}<f_i$ by Claim~\ref{claim1}.

\begin{cl}\label{claim2}
$\ex(n,F^*)=\left \lfloor \frac{d_i-1}{2}n \right \rfloor$.
\end{cl}
\begin{proof}
Since $i$ is the index maximizing the number of edges for $F$, thus for any $\ell>i$, $f_{\ell}\le f_{i}$, or $\ell-1+\frac{d_{\ell}-1}{2}\le i-1+\frac{d_i-1}{2}$, this implies
\begin{align}
\ell-i+\frac{d_{\ell}-1}{2}\le \frac{d_i-1}{2}. \label{eq:ell}
\end{align}
And by the induction hypothesis we have,
\begin{align*}
\ex(n,F^*)&=\max_{1\le i^*\le k-i+1}\left\{\left[i^*-1+\frac{d_{i^*}-1}{2}\right]n-\left\lfloor\frac{i^*-1}{2}(i^*+d_{i^*}-1)\right\rfloor\right\}\\
&\le \max_{1\le i^*\le k-i+1}\left\{\left[i^*-1+\frac{d_{i^*}-1}{2}\right]n\right\}\\
&= \max_{i\le \ell\le k}\left\{\left[\ell-i+\frac{d_{\ell}-1}{2}\right]n\right\}\\
&\le \frac{d_i-1}{2}n && \text{  by \eqref{eq:ell}.} \\
\end{align*}
On the other hand, a $S^i$-free graph is $F^*$-free, thus $\ex(n,F^*)\ge \left \lfloor \frac{d_i-1}{2}n \right \rfloor$.
\end{proof}

Let $G$ be an extremal graph for $F$ on $n$ vertices. Recall $F'=S^1\cup\cdots\cup S^{i-1}$, the choice of $i$ and the extremality of $G$ implies $e(G)\ge e(F(n,i))>\ex(n,F')$. Thus $F'\subseteq G$. As before if we can show there is a vertex with linear degree in each star in $F'\subseteq G$, then we have a set $U$ of order $i-1$, each vertex of which has linear degree and $G-U$ is $F^*$-free. Then by Claim~\ref{claim2} we get $G \subseteq F(n,i)$.

Take any star in a copy of $F'$ in $G$, say $S^j$, note that $j<i$. We take the same approach as before to prove that the number of edges between $S^j$ and $G-S^j$ is linear, namely $e_{0}=\Omega(n)$, which implies the existence of a vertex with linear degree as desired.

Note that $G-S^j$ must be $(F-S^j)$-free. We now give an upper bound on $e(G-S^j)$. Let $F''=F-S^j$  and let $i''$ be the index maximizing the number of edges for $F''$.

If $i''<j<i$, then by Claim~\ref{claim1}, $f_{i''}<f_i$ and thus $$e(G-S^j)\le \ex(n,F'')\le f_{i''}n.$$
Therefore $e_{0}=e(G)-e(G[S^j])-e(G-S^j)\ge f_in-f_{i''}n-{d_j+1\choose 2}=\Omega(n)$ as desired.

Notice that in $F''$, all indices after $j$ were shifted to the left by one, that is $F''=S^1\cup\cdots\cup S^{j-1}\cup S^{j+1}\cup\cdots\cup S^k$. Thus if $i''\ge j$, by the definition of $f_i$, then it is the same index as $F$ that maximizes the number of edges for $F''$, that is $i$ in $F$ and $i-1$ in $F''$.

Thus $i''=i-1$. In this case $e(G-S^j)\le [(i-1)-1]+\frac{d_i-1}{2}=f_i-1$. Thus $e_{0}\ge f_in-f_{i''}n-{d_j+1\choose 2}= n-{d_j+1\choose 2}=\Omega(n)$ as desired.

This finishes the proof of Case(2).

\noindent \textbf{Case (3):} $i=1$.

Let $G$ an extremal graph for $F$. We want to show: $e(G)=\ex(n,F)=\left \lfloor \frac{d_1-1}{2}n \right \rfloor$. Since an $S^1$-free graph is $F$-free, $e(G)\ge \left \lfloor \frac{d_1-1}{2}n\right \rfloor$.

We may assume $\Delta(G)\ge d_1$, since otherwise $e(G) \leq \frac{d_1-1}{2}n$. Let $v$ be a vertex of degree $\Delta(G)$, so $d(v)\ge d_1$. Thus we can get a $S^1$ from $N(v)\cup\{v\}$ with $v$ as its center. Note that since $i=1$, we have for any $j>1$, $f_j\le \frac{d_1-1}{2}$. Let $F^*=F-S^1=S^2\cup\cdots\cup S^k$. Note that in $F^*$, all indices were shifted to the left by one. Hence if $i^*$ is the index maximizing the number of edges for $F^*$, then it was $j=i^*+1\ge 2$ in $F$. Thus $f_{i^*}=[(j-1)-1]+\frac{d_j-1}{2}=f_{j}-1\le \frac{d_1-1}{2}-1$.

Let $e_{0}$ be the number of edges between $S^1$ and $G-S^1$, then $e_{0}=e(G)- e([S^1])-e(G-S^1)\ge \frac{d_1-1}{2}-{d_1+1\choose 2}-f_{i^*}n\ge n-{d_1+1\choose 2}=\Omega(n)$. Thus there is a vertex of linear degree in $S^1$, let it be $u$. This implies $G-\{u\}$ is $F^*$-free. Thus
$$e(G)=d(u)+e(G-u)\le n-1+\ex(n,F^*)\le n-1+f_{i^*}n\le \frac{d_1-1}{2}n-1<e(F(n,1)),$$
which is a contradiction.

This finishes the proof of Case (3) and hence also of Theorem~\ref{thm:stars}.
\qed


\section{Forests with components of order 4}

Now we consider the forest whose components are all of order 4. Notice that there are only two trees of order 4: the path $P_4$ and the star $S_3$. Let $F=a\cdot P_4 \cup b\cdot S_3$. Let $G_F^1(n)$ be the $n$-vertex graph constructed as follows: assume $n-b =3d+r$ with $r\le 2$, $G_F^1(n)$ contains $b$ universal vertices, the remaining graph is $K_r \cup d\cdot K_3$. Let $G_F^2(n)$ be the $n$-vertex graph containing $2a+b-1$ universal vertices and the remaining graph is empty. See Figure~\ref{fig-sp4}.
It is easy to check that both $G_F^1(n)$ and $G_F^2(n)$ are $F$-free as the set of universal vertices is too small.

\useplain{\begin{figure}[ht]}
\begin{center}
\includegraphics{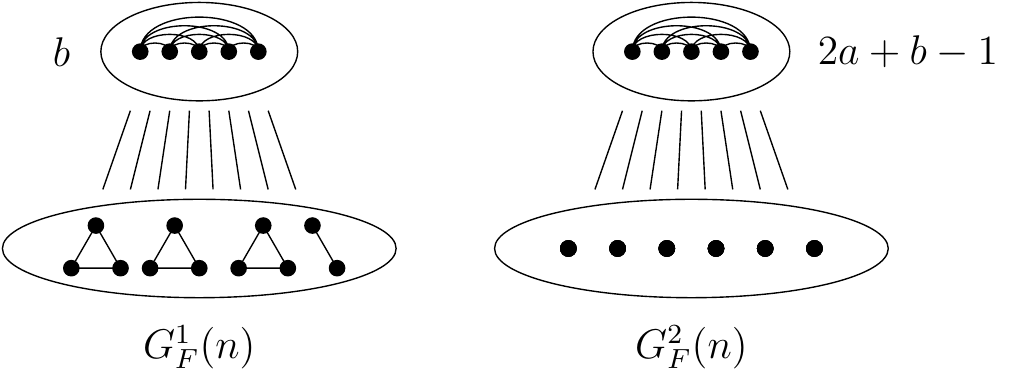}
\end{center}
\caption{The extremal graphs for  a forest with components of order 4.}
\label{fig-sp4}
\end{figure}

\begin{theorem}\label{order-4-case}
Given $F=a\cdot P_4 \cup b\cdot S_3$, and $n$ is sufficiently large and assume $n=3d+r$ with $r\le 2$, then
\begin{enumerate}
\renewcommand{\theenumi}{(\roman{enumi})}
\item If $a=1$ and $r=0$, then $G_F^1(n)$ is the unique extremal graph; if $a=1$ and $r\neq 0$, then $G_F^1(n)$ and $G_F^2(n)$ are the only extremal graphs.
\item If $a>1$, then $G_F^2(n)$ is the unique extremal graph.
\end{enumerate}
\end{theorem}

We only give a sketch of the proof. Use induction on $b$, the number of copies of $S_3$. Let $G$ be an extremal graph with $n$ vertices, then $e(G)\ge e(G_F^i(n))$, $i=1,2$. 
Thus $G$ contains a copy of $S_3$, similar as the proof for star-forest, any copy of $S_3$ in $G$ contains a vertex, say $v$, of degree $\Omega(n)$. 
Then $G-\{v\}$ must be $F'$-free where $F'=a\cdot P_4 \cup (b-1)\cdot S_3$. Then by the inductive hypothesis, $G-\{v\}$ has to be $G_{F'}^i(n-1)$, which then implies $G\subseteq G_F^i(n)$. 
The base case of the induction is when $b=0$, then the graph is $a \cdot P_4$. This also explains why we have two different constructions for the extremal graphs. 
Because when $a=1$, by a result of Faudree and Schelp~\cite{F-S}, vertex disjoint copies of triangles or a star are the extremal graphs (actually, they showed a combination of triangles and a smaller star is also an extremal case, but here, if any triangle appears, then the number of universal vertices will be fewer which yields a construction worse than $G_F^1$ and $G_F^2$); when $a>1$, then Theorem~\ref{disjoint-union-path} (or Theorem~\ref{thm:paths}) implies that $G_F^2(n)$ is the unique extremal graph.

The same technique can be applied on $a\cdot P_{\ell} \cup b\cdot S_t$ but the proof is very technical.

\section*{Acknowledgments}
We would like to thank J\'ozsef Balogh, Alexandr Kostochka and Douglas B. West for encouragement
and fruitful discussions.

\end{document}